\documentclass[letterpaper,12pt, reqno]{amsart}
\usepackage{amsmath, amsthm, amscd, amsfonts, amssymb, graphicx, color}
\usepackage[bookmarksnumbered, colorlinks, plainpages]{hyperref}
\usepackage{pgfplots}
\pgfplotsset{compat=1.15}
\usepackage{mathrsfs}
\usetikzlibrary{arrows}
\usetikzlibrary{positioning}

\hypersetup{colorlinks=true,linkcolor=red, anchorcolor=green,
citecolor=cyan, urlcolor=red, filecolor=magenta, pdftoolbar=true}

\newtheorem{theorem}{Theorem}
\newtheorem{lemma}[theorem]{Lemma}

\theoremstyle{definition}

\theoremstyle{remark}
\newtheorem{remark}[theorem]{Remark}
\numberwithin{equation}{section}

\DeclareMathOperator{\bd}{bd}

\usepackage{color}
\usepackage{enumerate}

\begin{document}

\setcounter{page}{1}

\title{Equilateral dimension of the planar Banach--Mazur compactum}

\author[T. Kobos \MakeLowercase{and} K. Swanepoel]
{Tomasz Kobos$^1$ \MakeLowercase{and} Konrad Swanepoel$^2$}
\address{$^1$Faculty of Mathematics and Computer Science, Jagiellonian University, \L oja\-sie\-wicza~6, 30-348 Krak\'ow, Poland}
\email{tomasz.kobos@uj.edu.pl}
\address{$^2$Department of Mathematics, London School of Economics and Political Science, Houghton Street, WC2A 2AE, London, UK}
\email{k.swanepoel@lse.ac.uk}

\begin{abstract}
We prove that there are arbitrarily large equilateral sets of planar and symmetric convex bodies in the Banach--Mazur distance. The order of the size of these $d$-equilateral sets asymptotically matches the bounds of the size of maximum-size $d$-separated sets (determined by Bronstein in 1978), showing that our construction is essentially optimal.
\end{abstract} \maketitle

\section{Introduction}

For two normed spaces $X, Y$ of the same dimension $n$ and over the same field $\mathbb{K}$ (where $\mathbb{K}=\mathbb{R}$ or $\mathbb{K}=\mathbb{C}$), a classical way of defining their distance originates from the work of Banach and Mazur. The \emph{Banach--Mazur distance} is defined as $\inf \|T\|\cdot \|T^{-1}\|,$ where infimum is taken over all linear and invertible operators $T: X \to Y$ and the standard operator norm is considered. It is easy to check that the infimum is attained by some operator $T$ and this is indeed a (multiplicative) distance, when considered on the set of all isometry classes of $n$-dimensional normed spaces. In other words, $d_{BM}(X, Y)=1$ if and only if $X$ and $Y$ are isometric. It is a classical fact that when the set of isometry classes of $n$-dimensional normed spaces is equipped with the Banach--Mazur distance, the resulting metric space is compact and is therefore traditionally called the \emph{Banach--Mazur compactum}. While the Banach--Mazur distance can be considered also for the complex scalars, we shall restrict our attention only to the real case.

A different way of looking at the Banach--Mazur distance in the real case is provided by the language of symmetric convex bodies in $\mathbb{R}^n$. For two origin-symmetric convex bodies $K, L \subseteq \mathbb{R}^n$ we can define their Banach--Mazur distance~as
\[d_{BM}(K,L) = \inf \{r >0 : K \subseteq T(L) \subseteq rK \},\] 
where the infimum is taken over all invertible linear operators $T : \mathbb{R}^n \to \mathbb{R}^n$. It is easy to check these two definitions agree, i.e.\ the equality $d_{BM}(X, Y)=d_{BM}(B_X, B_Y)$ holds for any two real $n$-dimensional normed spaces $X, Y$ and their unit balls.

The Banach--Mazur distance has been widely studied for several decades by many authors, in a variety of different contexts. For a systematic study of the Banach--Mazur distance and its role in the geometry of Banach spaces, we refer the reader to the classical monograph \cite{tomczak} on this subject. However, despite a large body of research on this topic, many basic questions about the Banach--Mazur compactum remain open. This is especially true for questions about exact values, as many of the known results were established only in the asymptotic setting. One example among many, is provided by the natural problem of determining the diameter of the Banach--Mazur compactum.
By John's Ellipsoid Theorem \cite{john}, $n$ is an upper bound on the diameter, and Gluskin \cite{gluskin} showed that this is asymptotically optimal, i.e.\ there exists an absolute constant $c>0$ such that the diameter is bounded from below by $cn$.
The only case where the diameter is exactly known (excluding the trivial $1$-dimensional case), is in the plane, where the diameter is equal to $\frac{3}{2}$ \cite{stromquist, lassak}. In higher dimensions, the exact value of the diameter is unknown.

Many challenges when working with the Banach--Mazur distance seem to be caused by the fact that it can be surprisingly difficult to compute it, even for concrete and familiar pairs of normed spaces such as the classical spaces $\ell_1^n$ and $\ell_{\infty}^n$ (for which the unit balls are the $n$-dimensional cross-polytope and the $n$-dimensional cube respectively). Obviously $d_{BM}(\ell_1^2, \ell_{\infty}^2)=1$ and for a general $n$ it has been known for a long time that the distance is asymptotically of the order $\sqrt{n}$. However, the precise value of the distance for $n=3$ and $n=4$ was determined only recently (see \cite{kobosvarivoda}) and is equal to $\frac{9}{5}$ and $2$ respectively. The exact value of the distance is not known in any higher dimension.

One classical way of evaluating the complexity of a given metric space $X$ is provided by the notion of the \emph{$r$-covering number} $N(X, r)$, which is the smallest possible number of open balls of radius $r > 0$ that cover $X$. A closely related concept is the \emph{$r$-packing number}, which gives the maximum possible cardinality of an $r$-separated set (i.e.\ a set in which every two elements are at distance at least $r$). Obviously, the packing number and covering numbers are always finite for any compact metric space (in particular for the Banach--Mazur compactum). Moreover, it is easy to see that for a given $r>0$, the $r$-packing number is at least $N(X, r)$ and not greater than $N(X, \frac{r}{2})$. Hence, in practice it is usually enough to estimate only one of the functions. 

The covering number of the Banach--Mazur compactum was studied already in the 1970s by Bronstein (see \cite{bronstein} and \cite{bronstein2}), who proved that for a fixed dimension $n$ and the distance $d=1+\varepsilon$, with $\varepsilon \to 0^{+}$, the covering number of the $n$-dimensional Banach--Mazur compactum has the order of $\exp(c\varepsilon^{\frac{1-n}{2}})$ (for some constant $c$ depending on $n$). This problem was later revisited by Pisier in \cite{pisier}, who instead obtained asymptotics of the covering number, when the distance is held fixed and dimension goes to infinity.

When one seeks to somehow capture the size of a given metric space $X$ with only a single number, rather than a function, one often considered parameter is the \emph{equilateral dimension} $e(X)$ of $X$. A \emph{$d$-equilateral set} in $X$ is a set, in which every two different points are at the distance exactly $d$ and the equilateral dimension $e(X)$ is defined as the maximum cardinality of a $d$-equilateral set for any possible $d$. For example, it is well known that $e(\ell_2^n)=n+1$, $e(\ell_{\infty}^{n})=2^n$, and $e(X)\leq 2^n$ for any $n$-dimensional Banach space. Problems of estimating the equilateral dimension of a given metric space has gained considerable attention and often turn out to be very challenging, even in the much more specific context of normed spaces. For example, it was conjectured by Kusner \cite{guy} that $e(\ell_1^n)=2n$ (with the obvious lower bound given by the vectors from the canonical unit basis and its negatives), but this was confirmed only for $n \leq 4$ (see \cite{bandelt} and \cite{koolen}), while the general weaker upper bound $e(\ell_1^n) \leq C n \log n$ was established in \cite{alon}. Perhaps the most famous open problem in this field was proposed by Petty in \cite{petty}, who conjectured that the inequality $e(X) \geq n+1$ holds for any normed space $X$ of dimension $n$. He proved it for $n \leq 3$ and later Makeev \cite{makeev} settled the case of $n=4$, while for $n \geq 5$ the conjecture remains open. We note that in the infinite-dimensional normed spaces, there are always arbitrarily large, finite equilateral sets, but there might not exist an infinite one (see \cite{terenzi} or \cite{koszmider}). However, every infinite dimensional uniformly smooth Banach space contains an infinite equilateral set \cite{freeman}. For a survey concerning equilateral sets in normed spaces, the reader is referred to \cite{swanepoelsurvey}.

Even if the equilateral sets were studied mostly in the context of normed spaces, they were considered for some other metric spaces as well. Some notable examples include: elliptic geometry \cite{elliptic}, Riemannian manifolds \cite{riemannian} and \cite{riemannian2}, Heisenberg group \cite{heisenberg} or the discrete hypercube with the Hamming distance \cite{hypercube}. Interestingly, one of the major open problems in the Quantum Information Theory, namely the Zauner conjecture about the existence of $n^2$ equiangular lines in $\mathbb{C}^n$, which is often frequently stated in terms of the existence of a SIC-POVM (a symmetric, informationally complete, positive-operator valued measure), can be also simply rephrased as $e(\mathbb{C}\mathbb{P}^{n-1})=n^2$, where $\mathbb{C}\mathbb{P}^{n-1}$ denotes the $n$-dimensional complex projective space endowed with the Fubini-Study metric. It should be noted that, when dealing with the equilateral sets in the normed spaces, the specific distance $d$ of a $d$-equilateral set plays no important role, due to the possibility of rescaling. However, it has to be taken into the account in general metric spaces.

To the best of our knowledge, equilateral sets in the Banach--Mazur compactum were not studied before. One notable $2$-equilateral set with three elements can be found in dimension $n=4$, since it is known for a long time that $d_{BM}(\ell_2^n, \ell_1^n)=d_{BM}(\ell_2^n, \ell_{\infty}^n)=\sqrt{n}$ for every $n$ and, as we have already mentioned before, it was established recently in \cite{kobosvarivoda} that also $d_{BM}(\ell_1^4, \ell_{\infty}^4)=2$. However, this seems to be rather coincidental and it is not immediately apparent how one should approach finding a $3$-element equilateral set already in the planar case.

As noted before, it follows immediately from compactness that the Banach--Mazur compactum does not posses any infinite equilateral set. Our main result states that the planar Banach--Mazur compactum contains arbitrarily large, finite equilateral sets. More specifically, we prove the following

\begin{theorem}
\label{mainthm}
For every sufficiently large integer $N$ there exists a $d^2_N$-equilateral set in the (symmetric) planar Banach--Mazur compactum with cardinality at least $C^N$, where $d_N= \frac{1}{\cos \frac{\pi}{4N}}$ and
$C=\left ( \frac{2^{18}}{2^{18}-1} \right )^{\frac{1}{20}} > 1.$
\end{theorem}

We recall that by the previously mentioned result of Bronstein, the cardinality of a $(1+\varepsilon)$-separated set in the planar Banach-compactum is bounded from above by $\exp(c\varepsilon^{-\frac{1}{2}})$. Surprisingly, this turns out to match the order of our estimate for the cardinality of an equilateral set. Indeed, if we write $d^2_N=1 + \varepsilon_N$, then $\varepsilon_N \sim \frac{\alpha}{N^2}$ for some constant $\alpha>0$ and therefore the constructed equilateral set has cardinality  of the order $C^{\sqrt{\varepsilon^{-1}_N}}$. Hence, our construction is of the best possible order. In other words, up to a constant inside the $\exp$ function, the maximum cardinality of a $d^2_N$-equilateral set is actually the same as of a $d^2_N$-separated set. This is a rather surprising property of the Banach--Mazur compactum, as there are many examples of metric spaces where separated sets are much larger than equilateral sets. For more details concerning the optimality of the result, see Remark~\ref{optimal}.

The proof of Theorem \ref{mainthm} is presented in Section \ref{sectproof}.
The main idea behind the construction is as follows. We consider a regular polygon with a large number of vertices inscribed in the Euclidean unit circle. The convex bodies contained in the equilateral set will be defined as a certain combination of the circle with the polygon, where between every two consecutive vertices we either choose the the arc of the circle or the segment of the regular polygon. Thus, every such convex body can be encoded by a binary sequence. We then impose certain combinatorial conditions on the binary sequences, which forces the corresponding convex bodies to already be in the optimal position, i.e.\ with the identity mapping realizing the Banach--Mazur distance. To prove that this is indeed the case, we carry out a detailed analysis of a general operator $T$ based on its orthogonal decomposition. We find an exponential number of binary sequences satisfying the desired conditions with a simple probabilistic argument.

Throughout the paper, by $\| x \|$ we will always denote the Euclidean norm of a vector $x \in \mathbb{R}^2$.

\section{Proof of the main result}
\label{sectproof}

Let $N \geq 1$ be an integer. We will treat sequences $a \in \{0, 1\}^N$ circularly, i.e.\ formally they are indexed by all integers, with the indices being periodic modulo $N$ (for any $i \in \mathbb{Z}$ we understand $a_i$ to be equal to $a_{i \mod{N}}$). By the term \emph{substring} we shall mean a sequence of consecutive positions contained in a given sequence (substrings will also be considered circularly). By a \emph{random sequence} from $\{0, 1\}^N$ we shall mean a sequence with coordinates drawn independently and with equal probability of drawing $0$ or $1$. We start with an easy lemma related to random sequences. We note that such problems are well-studied and much more precise results are known, but the simple estimate given below is sufficient for our needs.

\begin{lemma}
\label{lemsubseq}
Let $N \geq k \geq 1$ be positive integers and let $v \in \{0, 1\}^{k}$ be a fixed sequence. Then, the probability that a random sequence from $\{0, 1\}^N$ does not contain $v$ as a substring is at most $c_k^{\left \lfloor \frac{N}{k} \right \rfloor }$, where $c_k=1 - \frac{1}{2^k}$.
\end{lemma}

\begin{proof}
Let $N = km +r$ where $m=\lfloor N/k\rfloor$, and consider $m$ blocks of $k$ consecutive positions from $1$ to $km$. If a random sequence $a \in \{0, 1\}^N$ does not contain $v$ as a substring, then in particular, on every such block $a$ does not coincide with $v$, which happens with the probability $c_k$.
Thus, the probability that no block is equal to $v$ is equal to $c_k^m = c_k^{\left \lfloor \frac{N}{k} \right \rfloor }$. This is clearly an upper estimate for the probability that $v$ does not appear as a substring in $a$.
\end{proof}

In the proof of Theorem \ref{mainthm}, the constructed equilateral set will consist of convex bodies encoded by binary sequences. Our goal is to impose certain conditions on these sequences, which will force the corresponding convex bodies to already be in optimal Banach--Mazur position, i.e.\ so that the identity mapping will be the operator $T$ realizing the infimum in the definition of the Banach--Mazur distance.

Let $N$ be a (large) positive integer. Given two sequences $a, b \in \{0, 1\}^{4N}$, we will say that there are \emph{in balance}, if for every $k \in \mathbb{Z}$ there exist integers 
\[k+1 \leq i_1 < i_2 < i_3 \leq k+N-1\]
such that $i_{j+1}-i_{j} \geq 3$ for $1 \leq j \leq 2$ and
\[a_{i_1-1}=a_{i_1}=a_{i_1+1}=1, \quad b_{i_1-1} = b_{i_1} = b_{i_1+1} = 0, \]
\[a_{i_2-1}=a_{i_2}=a_{i_2+1}=0, \quad b_{i_2-1} = b_{i_2} = b_{i_2+1} = 1, \]
\[a_{i_3-1}=a_{i_3}=a_{i_3+1}=1, \quad b_{i_3-1} = b_{i_3} = b_{i_3+1} = 0, \]
and the same condition holds with $a$ and $b$ switched. For a sequence $a \in \{0, 1\}^{4N}$ and an integer $k$, by a \emph{rotation} by $k$ of $a$ we shall mean a sequence $b \in \{0, 1\}^{4N}$ such that $b_i=a_{i+k}$ for every $1 \leq i \leq 4N$. By the \emph{reflection} of $a$ we shall mean a sequence $b \in \{0, 1\}^{4N}$ such that $b_{i}=a_{1-i}$ for every $i$. A sequence $a \in \{0, 1\}^n$ will be called \emph{symmetric} if the sequence is periodic modulo $2N$, i.e.\ the equality $a_{2N+i}=a_i$ holds for every $1 \leq i \leq 4N$.

A family of sequences $\mathcal{S} \subseteq \{0, 1\}^{4N}$ will be called \emph{balanced} if for any two different sequences $a, b \in \mathcal{S}$, the sequence $a$ is in balance with every rotation of $b$ and also with every rotation of the reflection of $b$. Using a random construction, it is not difficult to find an exponential number of binary sequences forming a balanced family.

\begin{lemma}
\label{lemcard}
For every sufficiently large integer $N$, there exists a balanced family of symmetric sequences $\mathcal{S} \subseteq \{0, 1\}^{4N}$ of cardinality $\lceil C^{N} \rceil$, where
$C=\left ( \frac{2^{18}}{2^{18}-1} \right )^{\frac{1}{20}} > 1$.
\end{lemma}

\begin{proof}
We will provide a stronger construction than formally required, with all the desired positions considered being consecutive (i.e.\ $i_{j+1}=i_{j}+3$ for $1 \leq j \leq 2$). For a sequence $a \in \{0, 1\}^{2N}$, let $\widetilde{a} \in \{0, 1\}^{4N}$ be a symmetric sequence defined as  $\widetilde{a}_{2N+i}=\widetilde{a}_i=a_i$ for $1 \leq i \leq 2N$. Let us draw a family $\mathcal{F} \subseteq \{0, 1\}^{2N}$ of $n$ random sequences. We will establish an upper estimate of the probability that the family $\mathcal{S} = \{ \widetilde{a}: \ a \in \mathcal{F} \} \subseteq \{0, 1\}^{4N}$ is not balanced. We observe that, if $a \in \{0, 1\}^{2N}$ is a random sequence, then any substring of $\widetilde{a}$ of length $N-1$ is a random sequence from $\{0, 1\}^{N-1}$. Moreover, any rotation or a rotation of the reflection of a random sequence is still a random sequence. There are in total $8N$ different rotations or rotations of the reflection (i.e.\ mappings of the form $(a_i) \to (a_{\pm i + k}$), $n^2$ pair of sequences and $4N$ different possibilities of choosing a block of $N-1$ consecutive positions. Therefore, since the probability of a union is not greater than the sum of probabilities, we see that the probability that $\mathcal{S}$ fails to be balanced is at most 
\[8N \cdot 4N \cdot n^2 \cdot p = 32N^2 \cdot n^2 \cdot p,\]
where $p \in (0, 1)$ is the probability that we can not find the desired positions for two random sequences of length $N-1$. Equivalently, by making these two random sequences into one of length $2N-2$ (alternating their coordinates), $p$ is equal to the probability that a random sequence from $\{0, 1\}^{2N-2}$ does not contain the substring $101010010101101010.$
Thus, it follows from Lemma \ref{lemsubseq} that $p \leq c_{18}^{\left \lfloor \frac{N-1}{9} \right \rfloor}$, and we conclude that for any $n$ satisfying the inequality
\[32N^2 \cdot n^2 \cdot c_{18}^{\left \lfloor \frac{N-1}{9} \right \rfloor} < 1,\]
there exists a desired random family. Clearly we have
\[\left ( c_{18}^{-\frac{N}{20}} + 1 \right)^2 =c_{18}^{-\frac{N}{10}}+ 2 c_{18}^{-\frac{N}{20}} + 1 <  \frac{c_{18}^{-\left \lfloor \frac{N-1}{9} \right \rfloor}}{32N^2}\]
for all sufficiently large $N$. Thus, the inequality will be satisfied for $n  = \left \lceil C^{N} \right \rceil$, where 
\[C=c_{18}^{-\frac{1}{20}} = \left ( \frac{2^{18}}{2^{18}-1} \right )^{\frac{1}{20}}\]
and the conclusion follows.
\end{proof}

We need one more somewhat technical, but elementary, lemma.

\begin{lemma}
\label{lemineq}
For any real number $x \geq 1$ we have
\[2 \arccos \frac{1}{x} \geq \arccos \frac{2x^2}{x^4+1}.\]
\end{lemma}

\begin{proof}
Since $x \geq 1$ we have $0 < 2\arccos \frac{1}{x} \leq\pi$.
This, together with the fact that $\cos$ is decreasing on $[0,\pi]$, gives that the inequality is equivalent to $\cos\left ( 2\arccos \frac{1}{x} \right  ) \leq \frac{2x^2}{x^4+1}$.
By the cosine double angle formula, this is equivalent to $\frac{2}{x^2}-1 \leq\frac{2x^2}{x^4+1}$ for $x\geq 1$.
By subtracting $1$ from both sides, this is in turn rearranges to $\frac{2(x^2-1)}{x^2}\geq\frac{(x^2-1)^2}{x^4+1}$. 
After canceling out the non-negative factor $x^2-1$, we arrive after some manipulation at the inequality $x^4+x^2+2\geq 0$, which is clearly true.
\end{proof}
Now we are ready to prove Theorem \ref{mainthm}.

\begin{proof}[Proof of Theorem \ref{mainthm}]

Let $\mathbb{B} \subseteq \mathbb{R}^2$ be the Euclidean unit disc with the center in $0$ and let $\mathcal{C}_{4N} \subseteq \mathbb{B}$ be the regular $4N$-gon inscribed in $\mathbb{B}$ with a fixed set of vertices $\{x_1, \ldots, x_{4N}\} \subseteq \bd \mathbb{B}$. For any sequence $a \in \{0, 1\}^{4N}$ we define $K_{a} \subseteq \mathbb{R}^2$ to be the symmetric convex body which arises as a combination of the regular $4N$-gon $\mathcal{C}_{4N}$ and the unit disc $\mathbb{B}$, where $0$ corresponds to choosing a side of the regular $4N$-gon and $1$ to choosing an arc of the unit circle. More precisely, if  $\ell$ is any ray with endpoint $0$ which is contained between rays $x_i$ and $x_{i+1}$ for some $1 \leq i \leq 4N$, then $K_a \cap \ell= \mathcal{C}_{4N} \cap \ell$ when $a_i=0$, and $K_a \cap \ell= \mathbb{B} \cap \ell$ when $a_i=1$. In other words, between two consecutive vertices $x_i$ and $x_{i+1}$ the boundary of $K_a$ can be either the segment connecting these vertices or the arc of the unit disc connecting them. We will say that in this region $K_a$ is \emph{polygonal} or \emph{circular} respectively. The set $K_a$ is clearly convex and moreover, it is $0$-symmetric when the sequence $a$ is symmetric. Furthermore, if the linear transformation $U: \mathbb{R}^2 \to \mathbb{R}^2$ is a rotation by the angle $\alpha = k \frac{\pi}{2N}$ (where $k$ is an integer), then $U(K_a)=K_b$, where the sequence $b \in \{0, 1\}^{4N}$ is a rotation by $k$ of the sequence $a$. Similarly, if $U: \mathbb{R}^2 \to \mathbb{R}^2$ is a reflection across a line passing through the midpoint of a segment $[x_1x_{4N}]$, then $U(K_a)=K_b$, where $b \in \{0, 1\}^{4N}$ is the reflection of $a$ (as defined before). 

Our goal is to prove that if $\mathcal{S} \subseteq \{0, 1\}^{4N}$ is a balanced family of symmetric sequences, then for every $a, b \in \mathcal{S}$, $a \neq b$ we have $d_{BM}(K_a, K_b) = d_N^2$, where $d_N= \frac{1}{\cos \frac{\pi}{4N}}$ is the inverse of the distance of the origin to the midpoints of $\mathcal{C}_{4N}$. Then the desired conclusion will follow immediately from Lemma \ref{lemcard}. Throughout the proof, by $w_i \in \mathbb{R}^2$ we shall denote the midpoint of the segment $[x_ix_{i+1}]$ (thus $\|w_i\|=d_N^{-1}$). See Figure \ref{general} for an illustration of the construction.

\begin{figure}
\definecolor{qqffqq}{rgb}{0,0,0.9}
\definecolor{ffqqqq}{rgb}{0.9,0,0}
\definecolor{uuuuuu}{rgb}{0,0,0}
\label{general}
    \centering
\begin{tikzpicture}[line cap=round,line join=round,>=triangle 45,scale=2.7]
\clip(-1,-1.4) rectangle (1.6,1);
\draw [line width=1pt,color=ffqqqq] (-0.6250937417722248,-0.012550308498613438)-- (-0.6250937417722249,-0.5125503084986129);
\draw [line width=1pt,color=ffqqqq] (0.05791896011999467,-1.195563010390833)-- (0.5579189601199946,-1.195563010390833);
\draw [line width=1pt,color=ffqqqq] (1.2409316620122142,-0.5125503084986136)-- (1.240931662012214,-0.012550308498613771);
\draw [line width=1pt,color=ffqqqq] (0.057918960119994894,0.6704623933936058)-- (0.5579189601199948,0.6704623933936058);
\draw [line width=1pt,color=qqffqq] (-0.6250937417722249,-0.5125503084986129)-- (-0.375093741772225,-0.9455630103908327);
\draw [line width=1pt,color=qqffqq] (0.9909316620122143,0.4204623933936056)-- (1.240931662012214,-0.012550308498613771);
\draw [line width=1pt,color=qqffqq] (-0.37509374177222443,0.420462393393606)-- (0.057918960119994894,0.6704623933936058);
\draw [line width=1pt,color=qqffqq] (0.9909316620122139,-0.945563010390833)-- (0.5579189601199946,-1.195563010390833);
\draw [shift={(0.3079189601199948,-0.26255030849861344)},line width=1pt,color=qqffqq]  plot[domain=0.785398163397448:1.8325957145940461,variable=\t]({1*0.9659258262890681*cos(\t r)+0*0.9659258262890681*sin(\t r)},{0*0.9659258262890681*cos(\t r)+1*0.9659258262890681*sin(\t r)});
\draw [line width=1pt,dash pattern=on 2pt off 2pt] (0.3079189601199948,-0.26255030849861344) -- (0.3079189601199948,1.237923167385553);
\draw [line width=1pt,dotted,domain=0.3079189601199948:2.360881163177346] plot(\x,{(-0.35579013462529974--0.46650635094610937*\x)/0.8080127018922194});
\draw [line width=1pt,dash pattern=on 2pt off 2pt] (0.5579189601199948,0.6704623933936058)-- (0.9909316620122143,0.4204623933936056);
\draw [shift={(0.3079189601199948,-0.26255030849861344)},line width=1pt,color=ffqqqq]  plot[domain=0.7853981633974481:1.3089969389957472,variable=\t]({1*0.965925826289068*cos(\t r)+0*0.965925826289068*sin(\t r)},{0*0.965925826289068*cos(\t r)+1*0.965925826289068*sin(\t r)});
\draw [shift={(0.3079189601199948,-0.26255030849861344)},line width=1pt,color=ffqqqq]  plot[domain=0.26179938779914913:0.785398163397448,variable=\t]({1*0.9659258262890681*cos(\t r)+0*0.9659258262890681*sin(\t r)},{0*0.9659258262890681*cos(\t r)+1*0.9659258262890681*sin(\t r)});
\draw [line width=1pt,color=qqffqq] (1.2409316620122142,-0.5125503084986136)-- (0.9909316620122139,-0.945563010390833);
\draw [line width=1pt,color=ffqqqq] (1.2409316620122142,-0.5125503084986136)-- (0.9909316620122139,-0.945563010390833);
\draw [shift={(0.3079189601199948,-0.26255030849861344)},line width=1pt,dash pattern=on 2pt off 2pt]  plot[domain=5.497787143782138:6.021385919380437,variable=\t]({1*0.9659258262890683*cos(\t r)+0*0.9659258262890683*sin(\t r)},{0*0.9659258262890683*cos(\t r)+1*0.9659258262890683*sin(\t r)});
\draw [shift={(0.3079189601199948,-0.26255030849861344)},line width=1pt,color=qqffqq]  plot[domain=-0.26179938779914913:0.26179938779914913,variable=\t]({1*0.9659258262890685*cos(\t r)+0*0.9659258262890685*sin(\t r)},{0*0.9659258262890685*cos(\t r)+1*0.9659258262890685*sin(\t r)});
\draw [shift={(0.3079189601199948,-0.26255030849861344)},line width=1pt,color=ffqqqq]  plot[domain=4.974188368183839:5.497787143782138,variable=\t]({1*0.9659258262890685*cos(\t r)+0*0.9659258262890685*sin(\t r)},{0*0.9659258262890685*cos(\t r)+1*0.9659258262890685*sin(\t r)});
\draw [shift={(0.3079189601199948,-0.262550308498613)},line width=1pt,color=qqffqq]  plot[domain=3.9269908169872414:4.974188368183839,variable=\t]({1*0.9659258262890685*cos(\t r)+0*0.9659258262890685*sin(\t r)},{0*0.9659258262890685*cos(\t r)+1*0.9659258262890685*sin(\t r)});
\draw [shift={(0.3079189601199948,-0.26255030849861344)},line width=1pt,color=ffqqqq]  plot[domain=0.7853981633974481:1.3089969389957472,variable=\t]({-1*0.965925826289068*cos(\t r)+0*0.965925826289068*sin(\t r)},{0*0.965925826289068*cos(\t r)+-1*0.965925826289068*sin(\t r)});
\draw [line width=1pt,dash pattern=on 2pt off 2pt] (0.05791896011999487,-1.1955630103908326)-- (-0.37509374177222465,-0.9455630103908325);
\draw [shift={(0.3079189601199948,-0.26255030849861344)},line width=1pt,color=ffqqqq]  plot[domain=0.26179938779914913:0.785398163397448,variable=\t]({-1*0.9659258262890681*cos(\t r)+0*0.9659258262890681*sin(\t r)},{0*0.9659258262890681*cos(\t r)+-1*0.9659258262890681*sin(\t r)});
\draw [shift={(0.3079189601199948,-0.26255030849861344)},line width=1pt,color=qqffqq]  plot[domain=-0.26179938779914913:0.26179938779914913,variable=\t]({-1*0.9659258262890685*cos(\t r)+0*0.9659258262890685*sin(\t r)},{0*0.9659258262890685*cos(\t r)+-1*0.9659258262890685*sin(\t r)});
\draw [shift={(0.3079189601199948,-0.26255030849861344)},line width=1pt,dash pattern=on 2pt off 2pt]  plot[domain=5.497787143782138:6.021385919380437,variable=\t]({-1*0.9659258262890683*cos(\t r)+0*0.9659258262890683*sin(\t r)},{0*0.9659258262890683*cos(\t r)+-1*0.9659258262890683*sin(\t r)});
\draw [line width=1pt,color=qqffqq] (-0.6250937417722245,-0.012550308498613272)-- (-0.3750937417722243,0.42046239339360614);
\draw [line width=1pt,color=ffqqqq] (-0.37509374177222443,0.420462393393606)-- (-0.6250937417722248,-0.012550308498613438);
\draw [shift={(0.3079189601199948,-0.26255030849861344)},line width=1pt,color=ffqqqq]  plot[domain=4.974188368183839:5.497787143782138,variable=\t]({-1*0.9659258262890685*cos(\t r)+0*0.9659258262890685*sin(\t r)},{0*0.9659258262890685*cos(\t r)+-1*0.9659258262890685*sin(\t r)});
\draw [fill=uuuuuu] (0.05791896011999467,-1.195563010390833) circle (0.4pt) node[below] {$x_8$};
\draw [fill=uuuuuu] (0.5579189601199946,-1.195563010390833) circle (0.4pt) node[below right] {$x_7$};
\draw [fill=uuuuuu] (0.9909316620122139,-0.945563010390833) circle (0.4pt) node[below right] {$x_6$};
\draw [fill=uuuuuu] (1.2409316620122142,-0.5125503084986136) circle (0.4pt) node[right] {$x_5$};
\draw [fill=uuuuuu] (1.240931662012214,-0.012550308498613771) circle (0.4pt) node[right] {$x_4$};
\draw [fill=uuuuuu] (0.9909316620122143,0.4204623933936056) circle (0.4pt) node[right] {$x_3$};
\draw [fill=uuuuuu] (0.5579189601199948,0.6704623933936058) circle (0.4pt) node[below left=0mm and -2mm] {$x_2$};
\draw [fill=uuuuuu] (0.057918960119994894,0.6704623933936058) circle (0.4pt) node[above] {$x_1$};
\draw [fill=uuuuuu] (-0.37509374177222443,0.420462393393606) circle (0.4pt) node[above left] {$x_{12}$};
\draw [fill=uuuuuu] (-0.6250937417722248,-0.012550308498613438) circle (0.4pt) node[above left] {$x_{11}$};
\draw [fill=uuuuuu] (-0.6250937417722249,-0.5125503084986129) circle (0.4pt) node[left] {$x_{10}$};
\draw [fill=uuuuuu] (-0.375093741772225,-0.9455630103908327) circle (0.4pt) node[below left=-0.5mm and -0.5mm] {$x_{9}$};
\draw [fill=uuuuuu] (0.3079189601199948,-0.26255030849861344) circle (0.4pt) node[below right] {$0$};
\draw [fill=uuuuuu] (0.3079189601199948,0.6704623933936058) circle (0.4pt) node[below left=0mm and -1mm] {$w_1$};
\draw [fill=uuuuuu] (1.1159316620122142,0.2039560424474959) circle (0.4pt) node[above left=-1.5mm and 0.5mm] {$w_3$};
\draw [fill=uuuuuu] (0.3079189601199948,0.7033755177904546) circle (0.4pt) node[above right=-0.5mm and -0.5mm] {$d_Nw_1$};
\draw [fill=uuuuuu] (1.1444352638578026,0.22041260464592027) circle (0.4pt) node[right=1mm] {$d_Nw_3$};
\end{tikzpicture}
   
    \caption{An example of two convex bodies $K_a$ (colored in blue) and $K_b$ (colored in red), which arise as certain combinations of the regular $12$-gon with the Euclidean unit disc.
    For the purpose of clarity of the drawing, an illustration is shown for small $N$ equal to $3$. In this case, the presented convex bodies are not in balance according to our definition. Nevertheless, looking at the dashed ray through $w_1$ we see that $r=d_N$ is the smallest positive dilatation factor such that $K_a \subseteq rK_b$. The same holds true for the inclusion $K_b \subseteq rK_a$, as shown by the dotted line passing trough $w_3$.}
\end{figure}
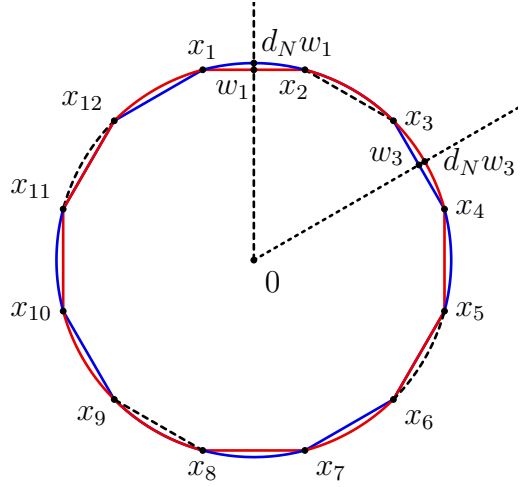

For any two convex bodies $K, L \subseteq \mathbb{R}^2$ let us consider the scaling factor $r(K, L)$ which is defined as the smallest dilatation factor $r>0$ such that $L \subseteq r K$. We need to prove that the inequality 
\begin{equation}
r(K_a, T(K_b)) \cdot r(T(K_b), K_a) \geq d_N^2
\end{equation}
holds for any invertible linear transformation $T: \mathbb{R}^2 \to \mathbb{R}^2$. By orthogonal decomposition, the operator $T$ can be written in the form $T=PU$, where $P: \mathbb{R}^2 \to \mathbb{R}^2$ is a positive definite linear operator and $U: \mathbb{R}^2 \to \mathbb{R}^2$ is an orthogonal operator. Our approach is to start with the most special case and make it more general in each subsequent step, arriving finally at the case of a completely general linear operator $T$. Therefore, we start our investigation with the case of $T$ being the identity mapping. In this situation, it is easy to observe that for any $a, b \in \mathcal{S}, a \neq b$ we clearly have $r(K_a, K_b)=r(K_b, K_a)=d_N$ (as in Figure \ref{general}). Indeed, since the sequences $a$ and $b$ are in balance by the assumption, in particular there exist integers $i, j$ such that $a_i=0, b_i=1$ and $a_j=1, b_j=0$. Hence, in the region between $x_i$ and $x_{i+1}$ the convex body $K_a$ is polygonal, while $K_b$ is circular. Thus, the midpoint $w_i$ of the segment $[x_ix_{i+1}]$ belongs to the boundary of $K_a$ and is of Euclidean norm $\frac{1}{d_N}$, so it requires a scaling factor of at least $d_N$ to cover the corresponding point of $K_b$ (of the Euclidean norm $1$). This shows that $r(K_a, K_b) \geq d_N$. On the other hand, the opposite inequality is obvious as for any point $x \in \bd K_a$ or $x \in \bd K_b$ we have $\frac{1}{d_N} \leq \|x\| \leq 1$, so that $K_b \subseteq \mathbb{B} \subseteq d_N K_a$. This shows that $r(K_a, K_b)=d_N$ and similarly we have $r(K_b, K_a)=d_N$.

Let us now consider the case of $T=U$ being an orthogonal transformation. It will turn out that in this situation we still have $r(K_a, U(K_b))=r(U(K_b), K_a)=d_N$. An orthogonal transformation $U$ can be either a rotation or a reflection. Let us start with assuming that $U$ is a rotation by an angle $\alpha$. We note that if $\alpha$ is a multiple of $\frac{\pi}{2N}$, i.e.\ $\alpha = \frac{k \pi}{2N}$ for $k \in \mathbb{Z}$, then the previous reasoning can be still applied, yielding that $r(K_a, U(K_b))=r(U(K_b), K_a)=d_N$. Indeed, in this case $U(K_b)=K_c$, where $c \in \{0, 1\}^{4N}$ is a rotation of $b$ by $k$ and, by the assumption, the sequence $a$ is in balance not only with $b$ but also with every rotation of $b$. Now, let us suppose that $U$ is a general rotation, so that $\alpha = k \frac{\pi}{2N} + \beta$, where $\beta \in \left [-\frac{\pi}{4N}, \frac{\pi}{4N} \right ]$. Again, since $a$ is in balance with a sequence $c$ that is a rotation of $b$ by $k$, we can find integers $i, j$ such that $a_i=0, c_i=1$ and $a_j=1, c_j=0$. Moreover, we have $U(K_b) = \widetilde{U}(K_c)$, where $\widetilde{U}: \mathbb{R}^2 \to \mathbb{R}^2$ is a rotation by $\beta$. In this case, for $w_j$ the midpoint of a segment $[x_jx_{j+1}]$ we have $w_j \in \bd K_c$ and thus $\widetilde{U}(w_j) \in \bd \widetilde{U}(K_c)= \bd U(K_b)$. Moreover $\|\widetilde{U}(w_j)\|=\|w_j\|=d_N^{-1}$. However, the ray starting in the origin and passing through $\widetilde{U}(w_j)$ (with an angle $\beta$ to the ray passing through $w_j$), still intersects the arc connecting $x_j$ and $x_{j+1}$, by the fact that $\beta \in \left [-\frac{\pi}{4N}, \frac{\pi}{4N} \right ]$. In this region, the convex body $K_a$ is circular, so that a scaling factor of $d_N$ is needed for $\widetilde{U}(w_j) \in \bd U(K_b)$ to cover the corresponding point $d_N\widetilde{U}(w_j) \in \bd K_a$ (see Figure \ref{orthogonal1}). This shows that $r(U(K_b), K_a) \geq d_N$ and the opposite inequality is again immediate, since $U$ is an isometry. 

\begin{figure}
\label{orthogonal1}
    \centering
\definecolor{qqwuqq}{rgb}{0,0,0.9}
\definecolor{qqffqq}{rgb}{0,0,0.9}
\definecolor{ffqqqq}{rgb}{0.9,0,0}
\definecolor{uuuuuu}{rgb}{0,0,0}
\begin{tikzpicture}[line cap=round,line join=round,>=triangle 45,scale=10.5]
\clip(0.01,0.2858079228996136) rectangle (0.64,0.77);
\draw [shift={(0.3079189601199948,0.3130246193646485)},line width=1pt,color=qqwuqq,fill=qqwuqq,fill opacity=0.10000000149011612] (0,0) -- (90:0.04) arc (90:116.6148378209898:0.04) -- cycle;
\draw [line width=1pt,color=ffqqqq] (0.057918960119994894,0.6704623933936058)-- (0.5579189601199948,0.6704623933936058);
\draw [shift={(0.3079189601199948,-0.26255030849861344)},line width=1pt,color=qqffqq]  plot[domain=1.3089969389957472:1.8325957145940461,variable=\t]({1*0.9659258262890681*cos(\t r)+0*0.9659258262890681*sin(\t r)},{0*0.9659258262890681*cos(\t r)+1*0.9659258262890681*sin(\t r)});
\draw (0.46,0.64) node {$K_c$};
\draw (0.47,0.72) node {$K_a$};
\draw [line width=1pt,dash pattern=on 2pt off 2pt,domain=-0.07400908347577384:0.3079189601199948] plot(\x,{(-0.14852376820901747--0.3195630424093731*\x)/-0.16012877389414154});
\draw [line width=1pt] (0.3079189601199948,0.3130246193646485) -- (0.3079189601199948,0.7490708414534015);
\draw (0.29,0.4) node {$\beta$};
\draw [fill=uuuuuu] (0.5579189601199948,0.6704623933936058) circle (0.15pt) node[right] {$x_{j+1}$};
\draw [fill=uuuuuu] (0.057918960119994894,0.6704623933936058) circle (0.15pt) node[left] {$x_j$};
\draw [fill=uuuuuu] (0.3079189601199948,0.6704623933936058) circle (0.15pt) node[below right] {$w_j$};
\draw [fill=uuuuuu] (0.3079189601199948,0.7033755177904546) circle (0.15pt) node[above right=0mm and 0mm] {$d_Nw_j$};
\draw [fill=uuuuuu] (0.3079189601199948,0.3130246193646485) circle (0.15pt) node[right] {$0$};
\draw [fill=uuuuuu] (0.14779018622585327,0.6325876617740216) circle (0.15pt) node[right=0.5mm] {$\widetilde{U}(w_j)$};
\draw [fill=uuuuuu] (0.12142618522724069,0.6852013186722501) circle (0.15pt) node[above right=0.9mm and -2.1mm] {$d_N\widetilde{U}(w_j)$};
\end{tikzpicture}
    \caption{In the region between $x_j$ and $x_{j+1}$ the convex body $K_a$ is circular (colored in blue), while $K_c$ is polygonal (colored in red). Since  $\beta \in \left [-\frac{\pi}{4N}, \frac{\pi}{4N} \right ]$, the ray through $\widetilde{U}(w_j)$ is still in this region and the ratio between points from $\bd K_a$ and $\bd U(K_b)$ on this ray is equal to $d_N$.}
\end{figure}
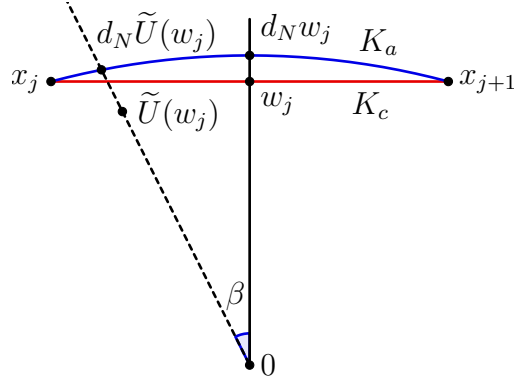

To establish the inequality $r(K_a, U(K_b)) \geq d_N$ we argue in a similar way, observing that for $w_i$ (the midpoint of the segment $[x_ix_{i+1}]$) the ray passing through $\widetilde{U}^{-1}(w_i)$ (angled to the ray passing through $w_i$ by $-\beta$) intersects the arc $[x_ix_{i+1}]$, where the convex body $K_c$ is circular, while $K_a$ is polygonal. Thus, to cover $d_Nw_i=\widetilde{U}(\widetilde{U}^{-1}(d_N w_i) \in \bd(\widetilde{U}(K_c))=\bd(U(K_b))$ by a corresponding point $w_i \in \bd K_a$ we need a scaling factor of $d_N$. It follows that $r(K_a, U(K_b))=r(U(K_b), K_a)=d_N$.

\begin{figure}
\label{orthogonal2}
    \centering
\definecolor{qqwuqq}{rgb}{0,0,0.9}
\definecolor{ffqqqq}{rgb}{0.9,0,0}
\definecolor{qqffqq}{rgb}{0,0,0.9}
\definecolor{uuuuuu}{rgb}{0,0,0}
\begin{tikzpicture}[line cap=round,line join=round,>=triangle 45,scale=9.5]
\clip(0,0.3) rectangle (0.72,0.78);
\draw [shift={(0.3079189601199948,0.3130246193646485)},line width=1pt,color=qqwuqq,fill=qqwuqq,fill opacity=0.10000000149011612] (0,0) -- (63.38516217901019:0.04) arc (63.38516217901019:90:0.04) -- cycle;
\draw [line width=1pt,color=qqffqq] (0.057918960119994894,0.6704623933936058)-- (0.5579189601199948,0.6704623933936058);
\draw [shift={(0.3079189601199948,-0.26255030849861344)},line width=1pt,color=ffqqqq]  plot[domain=1.3089969389957472:1.8325957145940461,variable=\t]({1*0.9659258262890681*cos(\t r)+0*0.9659258262890681*sin(\t r)},{0*0.9659258262890681*cos(\t r)+1*0.9659258262890681*sin(\t r)});
\draw (0.2,0.727) node {$K_c$};
\draw (0.2,0.64) node {$K_a$};
\draw [line width=1pt] (0.3079189601199948,0.3130246193646485) -- (0.3079189601199948,0.9075953441790285);
\draw (0.325,0.39) node {$\beta$};
\draw [line width=1pt,dash pattern=on 2pt off 2pt,domain=0.30791896011999487:0.6388300006341741] plot(\x,{(-0.04827527121393446--0.3195630424093731*\x)/0.1601287738941416});
\draw [fill=uuuuuu] (0.5579189601199948,0.6704623933936058) circle (0.15pt) node[below right=-0.5mm and -0.2mm] {$x_{i+1}$};
\draw [fill=uuuuuu] (0.057918960119994894,0.6704623933936058) circle (0.15pt) node[below left=-0.5mm and -0.5mm] {$x_i$};
\draw [fill=uuuuuu] (0.3079189601199948,0.6704623933936058) circle (0.15pt) node[below right] {$w_i$};
\draw [fill=uuuuuu] (0.3079189601199948,0.7033755177904546) circle (0.15pt) node[above right] {$d_Nw_i$};
\draw [fill=uuuuuu] (0.3079189601199948,0.3130246193646485) circle (0.15pt) node[right] {$0$};
\draw [fill=uuuuuu] (0.46804773401413646,0.6325876617740216) circle (0.15pt) node[below right=-1.2mm and -1mm] {$\widetilde{U}^{-1}(w_i)$};
\draw [fill=uuuuuu] (0.49441173501274904,0.6852013186722501) circle (0.15pt) node[above right=-1.5mm and 1.2mm] {$\widetilde{U}^{-1}(d_Nw_i)$};
\end{tikzpicture}
    \caption{In the region between $x_i$ and $x_{i+1}$ the convex body $K_a$ is polygonal (colored in blue), while $K_c$ is circular (colored in red). Since  $\beta \in \left [-\frac{\pi}{4N}, \frac{\pi}{4N} \right ]$, the ray through $\widetilde{U}^{-1}(w_i)$ is still in this region, so that $d_Nw_i=\widetilde{U}(\widetilde{U}^{-1}(d_N w_i) \in \bd U(K_b)$ and the ratio between points from $\bd U(K_b)$ and $K_a$ on the ray through $w_i$ is equal to $d_N$.}
\end{figure}
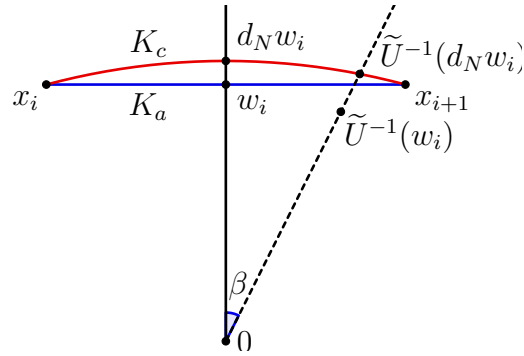

With essentially the same argument we can handle the case of $U$ being a reflection. Clearly, any reflection can be written as a composition of a reflection in the perpendicular bisector of the segment $[x_1x_{4N}]$ and an appropriate rotation. However, the reflection of $K_b$ in the perpendicular bisector of the segment $[x_1x_{4N}]$ is equal to $K_{\widetilde{b}}$, where $\widetilde{b} \in \{0, 1\}^{4N}$ is the reflection of the sequence $b$ defined before Lemma~\ref{lemcard}. Since, by the assumption, $a$ is in balance also with every rotation of $\widetilde{b}$, the previous reasoning can be still applied.

We move to the general case of $T=PU$, where $P: \mathbb{R}^2 \to \mathbb{R}^2$ is a positive definite linear operator and $U: \mathbb{R}^2 \to \mathbb{R}^2$ is an orthogonal operator. Only here the full strength of the conditions defining the relation of two sequences being in balance comes into play. Let us suppose that for $T=PU$ we get a smaller Banach--Mazur distance, so that
\begin{equation}
\label{ineq}
r(K_a, PU(K_b)) \cdot r(PU(K_b), K_a) < d_N^2.
\end{equation}
A positive definite transformation $P$ has two eigenvectors $u, v \in \mathbb{R}^2$, which form an orthonormal basis of $\mathbb{R}^2$, i.e.\ $\|u\|=\|v\|=1$ and $\langle u, v \rangle = 0$. Let $A \geq B$ be the corresponding positive eigenvalues of $P$. Since any rescaling of the operator $T$ yields the same bound on the Banach--Mazur distance, we can further suppose that $B=1$. If $A=1$, then $P$ is the identity and the situation reduces to the case of $T=U$ being an orthogonal transformation, which has already been settled. Let us thus suppose that $A>1$. We shall first obtain an upper bound on $A$ using \eqref{ineq}.

Obviously for any $x \in \bd K_a$ or $x \in \bd U(K_b)$ we have $\frac{1}{d_N} \leq \|x\| \leq 1$. As $\|P(u)\|=A\|u\|=A$, we get that $r(K_a, PU(K_b)) \geq \frac{A}{d_N}.$ On the other hand, as $\|P(v)\|=\|v\|=1$, we have $r(PU(K_b), K_a) \geq \frac{1}{d_N}$. Therefore by \eqref{ineq} we get
\[d_N^2 > r(K_a, PU(K_b)) \cdot r(PU(K_b), K_a) \geq \frac{A}{d_N^2},\]
which yields $A < d_N^4$.

Using this upper bound on $A$ we shall now estimate how much the operator $P$ can vary the direction of a given line through the origin. In other words, we are interested in the maximum of $\arccos\frac{\langle x, P(x) \rangle}{\|P(x)\|}$, where $x \in \mathbb{R}^2$ is a unit vector. We claim that this maximum is not larger than $\frac{\pi}{2N}$. Indeed, we note first that if $x=x_1u + x_2v$, then $\langle x, P(x) \rangle = Ax^2_1 + x_2^2$ is positive. Moreover, we have
\[\frac{\langle x, P(x) \rangle}{\|P(x)\|} = \frac{Ax^2_1 + x_2^2}{\sqrt{A^2x_1^2 + x_2^2}}=\frac{(A-1)x_1^2 + 1}{\sqrt{(A^2-1)x_1^2 + 1}}.\]
The function $f:[-1, 1] \to \mathbb{R}$ defined as
\[f(t) = \frac{(A-1)t^2 + 1}{\sqrt{(A^2-1)t^2 + 1}}\]
has derivative equal to
\[f'(t) = \frac{(A-1)^2 t((A+1)t^2 - 1)}{\sqrt{((A^2-1)t^2 + 1)^3}},\]
so that a minimum of $f$ on $[-1, 1]$ occurs at $t = \pm \frac{1}{\sqrt{A+1}}$ and is equal to $\frac{2 \sqrt{A}}{A+1}.$ As this is increasing in $A$, we can use the bound $A < d_N^4$ to obtain
\[\arccos\frac{\langle x, P(x) \rangle}{\|P(x)\|} \leq \arccos \frac{2d_N^2}{d_N^4+1}\]
for any unit vector $x \in \mathbb{R}^2$. Moreover, since  $d_N = \cos \frac{\pi}{4N}^{-1}$ we have $\frac{\pi}{2N} = 2 \arccos \frac{1}{d_N}$. Using the bound above combined with Lemma \ref{lemineq} we get
\[\arccos\frac{\langle x, P(x) \rangle}{\|P(x)\|} \leq \arccos \frac{2d_N^2}{d_N^4+1} \leq 2 \arccos \frac{1}{d_N} = \frac{\pi}{2N}.\]
In this way we have proved that any ray through the origin is rotated by $P$ by an angle not greater than $\frac{\pi}{2N}$. This is crucial, since our construction is carried out on arcs exactly of this measure. Hence, if a ray $\ell$ through the origin is contained between rays $x_{i}$ and $x_{i+1}$ for some integer $i$, then the image $P(\ell)$ is contained between the rays $x_{i-1}$ and $x_{i+2}$.

This is why there are blocks of three consecutive positions in the definition of sequences being in balance.

As before,
the operator $U$ is a rotation by angle $\alpha = k \frac{\pi}{2N} + \beta$, where $\beta \in \left [-\frac{\pi}{4N}, \frac{\pi}{4N} \right ]$, composed possibly with a reflection in the perpendicular bisector of the segment $[x_1x_{4N}]$. Let $c \in \{0, 1\}^{4N}$ be the corresponding sequence that is either a rotation by $k$ of $b$ or a rotation of the reflection of $b$ if $U$ involves the reflection. Let integer $1 \leq j \leq 4N$ be such that $u$ lies between $x_j$ and $x_{j+1}$ on the unit circle. Since the vectors $u$ and $v$ are perpendicular, $v$ lies between $x_{j+N}$ and $x_{j+N+1}$ or between $x_{j-N}$ and $x_{j-N+1}$. Let indices $i_1, i_2, i_3 \in I$ be as in the definition of being in balance for $a$ and $c$, where a set $I$ consisting of $N-1$ consecutive integers is equal to $\{j+1, j+2, \ldots, j+N-2 \}$ in the first case (when $v$ lies between $x_{j+N}$ and $x_{j+N+1}$) or to $\{j-N+1, j-N+2, \ldots, j-1 \}$ in the second case (when $v$ lies between $x_{j-N}$ and $x_{j-N+1}$). 

From $\eqref{ineq}$ it follows that $r(K_a, PU(K_b)) < d_N$  or $r(PU(K_b), K_a) < d_N$. First we will prove that the former inequality can not hold. By the assumption, the boundary of convex body $K_c$ contains a whole arc of the unit circle between the points $x_{i_2-1}, x_{i_2+2}$, while the boundary of the convex body $K_a$ is a union of three segments: $[x_{i_2-1}x_{i_2}]$, $[x_{i_2}x_{i_2+1}]$ and $[x_{i_2+1}x_{i_2+2}]$. Therefore $w_{i_2} \in \bd K_a$ and obviously $\|w_{i_2}\|=d_N^{-1}$. Let $\ell$ be the ray through $w_{i_2}$ starting at the origin and let $\ell_1$ be the ray starting in the origin such that $P(\ell_1)=\ell$. Because we have shown that $P$ rotates every line by no more than $\frac{\pi}{2N}$, it follows that the angle between $\ell_1$ and $\ell$ is not greater than $\frac{\pi}{2N}$. Now let $\widetilde{U}: \mathbb{R}^2 \to \mathbb{R}^2$ be the rotation by $\beta$ and let $\ell_2$ be the ray with an angle to $\ell_1$ equal to $-\beta$, i.e.\ $\widetilde{U}(\ell_2)=\ell_1$. The angle between $\ell$ and $\ell_2$ does not exceed $|\beta| + \frac{\pi}{2N} \leq \frac{3\pi}{4N}$. In particular, it follows that the ray $\ell_2$ is contained between the rays passing through $x_{i_2-1}$ and $x_{i_2+2}$. We recall that in this region the boundary of $K_c$ is circular. Thus, let $y \in \bd K_c \cap \ell_2$ be a unit vector. Then by construction, $z=P(\widetilde{U}(y)) \in \bd P(U(K_b)) \cap \ell$, and since $\widetilde{U}$ is an isometry and $P$ never decreases the Euclidean norm of a given vector, it follows that
\begin{equation}
\label{ratio}
\|z\| = \|P(\widetilde{U}(y))\| \geq \|\widetilde{U}(y)\| = \|y\|=1.
\end{equation}
Hence, on the ray $\ell$ the ratio between points of the boundaries of $P(U(K_b))$ and $K_a$ is equal to $\frac{\|z\|}{\|w_{i_2}\|} \geq d_N$. This proves that $r(K_a, PU(K_b)) \geq d_N$.

\begin{figure}
\label{general1}
    \centering
\definecolor{qqwuqq}{rgb}{0,0,0.9}
\definecolor{ffqqqq}{rgb}{0.9,0,0}
\definecolor{qqffqq}{rgb}{0,0,0.9}
\definecolor{uuuuuu}{rgb}{0,0,0}
\begin{tikzpicture}[line cap=round,line join=round,>=triangle 45,scale=5]
\clip(-0.47,-0.3) rectangle (1.1,0.88);
\draw [shift={(0.3079189601199948,-0.26255030849861344)},line width=1pt,color=qqwuqq,fill=qqwuqq,fill opacity=0.10000000149011612] (0,0) -- (55.782580362115304:0.2) arc (55.782580362115304:67.48671809600299:0.2) -- cycle;
\draw [line width=1pt,color=qqffqq] (0.057918960119994894,0.6704623933936058)-- (0.5579189601199948,0.6704623933936058);
\draw [shift={(0.3079189601199948,-0.26255030849861344)},line width=1pt,color=ffqqqq]  plot[domain=1.3089969389957472:1.8325957145940461,variable=\t]({1*0.9659258262890681*cos(\t r)+0*0.9659258262890681*sin(\t r)},{0*0.9659258262890681*cos(\t r)+1*0.9659258262890681*sin(\t r)});
\draw (-0.23,0.7) node[anchor=north west] {$K_c$};
\draw (-0.15846148808846763,0.57) node[anchor=north west] {$K_a$};
\draw (0.395,0.04) node[anchor=north west] {$\beta$};
\draw [line width=1pt,color=qqffqq] (-0.37509374177222443,0.420462393393606)-- (0.057918960119994894,0.6704623933936058);
\draw [line width=1pt,color=qqffqq] (0.9909316620122143,0.4204623933936056)-- (0.5579189601199948,0.6704623933936058);
\draw [shift={(0.3079189601199948,-0.26255030849861344)},line width=1pt,color=ffqqqq]  plot[domain=1.8325957145940461:2.356194490192345,variable=\t]({1*0.9659258262890681*cos(\t r)+0*0.9659258262890681*sin(\t r)},{0*0.9659258262890681*cos(\t r)+1*0.9659258262890681*sin(\t r)});
\draw [shift={(0.3079189601199948,-0.26255030849861344)},line width=1pt,color=ffqqqq]  plot[domain=0.7853981633974481:1.3089969389957472,variable=\t]({1*0.965925826289068*cos(\t r)+0*0.965925826289068*sin(\t r)},{0*0.965925826289068*cos(\t r)+1*0.965925826289068*sin(\t r)});
\draw [line width=1pt] (0.3079189601199948,-0.26255030849861344) -- (0.3079189601199948,1.1192729355626976);
\draw [line width=1pt,dash pattern=on 2pt off 2pt,domain=0.3079189601199948:1.44261518221614] plot(\x,{(-0.371864618328392--0.8923133886314669*\x)/0.36985067035392455});
\draw [line width=1pt,dotted,domain=0.3079189601199948:1.44261518221614] plot(\x,{(-0.38855557918108685--0.7987333829075265*\x)/0.543173715234194});
\draw (0.235,0.2) node[anchor=north west] {$\ell$};
\draw (0.37,0.2) node[anchor=north west] {$\ell_1$};
\draw (0.54,0.12) node[anchor=north west] {$\ell_2$};
\draw [fill=uuuuuu] (0.9909316620122143,0.4204623933936056) circle (0.25pt) node[below] {$x_{i_2+2}$};
\draw [fill=uuuuuu] (0.5579189601199948,0.6704623933936058) circle (0.25pt) node[below left=0mm and -1.5mm] {$x_{i_2+1}$};
\draw [fill=uuuuuu] (0.057918960119994894,0.6704623933936058) circle (0.25pt) node[below right=0mm and -0.5mm] {$x_{i_2}$};
\draw [fill=uuuuuu] (-0.37509374177222443,0.420462393393606) circle (0.25pt) node[below] {$x_{i_2-1}$};
\draw [fill=uuuuuu] (0.3079189601199948,-0.26255030849861344) circle (0.25pt) node[right] {$0$};
\draw [fill=uuuuuu] (0.3079189601199948,0.6704623933936058) circle (0.25pt) node[below left=0mm and -0.5mm] {$w_{i_2}$};
\draw [fill=uuuuuu] (0.3079189601199948,0.7033755177904546) circle (0.25pt);
\draw [fill=uuuuuu] (0.6777696304739194,0.6297630801328535) circle (0.25pt) node[above right=-1.5mm and 1mm] {$\widetilde{U}(y)$};
\draw [fill=uuuuuu] (0.8510926753541888,0.5361830744089131) circle (0.25pt) node[above right=-1.5mm and 1mm] {$y$};
\draw [fill=uuuuuu] (0.3079189601199948,0.7465950676551087) circle (0.25pt) node[above right=-1.5mm and -0.5mm] {$z=P(\widetilde{U}(y))$};
\end{tikzpicture}

    \caption{In the region between $x_{i_2-1}$ and $x_{i_2+2}$ the convex body $K_a$ is polygonal (colored in blue), while $K_c$ is circular (colored in red). Since the operator $P$ does not change the direction of any line by more than $\frac{\pi}{2N}$, and $\beta \in \left [-\frac{\pi}{4N}, \frac{\pi}{4N} \right ]$, the ray $\ell_2$ intersecting the unit circle in $y$ is still in the region between $x_{i_2-1}$ and $x_{i_2+2}$. It follows that $z = P(\widetilde{U}(y))$ is a boundary point of $P(U(K_b))$, so that the ratio between points from $\bd P(U(K_b))$ and $K_a$ on the ray $\ell$ is at least $d_N$.}
\end{figure}
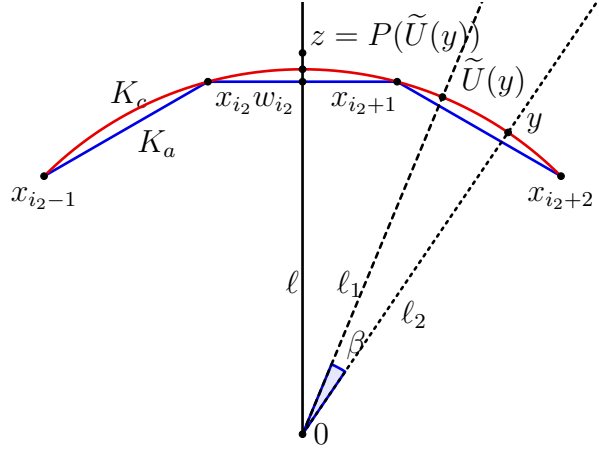

We conclude that we must have $r(PU(K_b), K_a) < d_N$. Let us thus assume that $r(PU(K_b), K_a) = s \cdot d_N$, where $s \in (0, 1)$. Arguing similarly to in the previous step, we shall prove that in this case we actually have $r(K_a, PU(K_b)) \geq \frac{d_N}{s}$, which will clearly contradict the inequality \eqref{ineq}.

Reasoning as in the previous step, we see that the image $P(\widetilde{U}(w_{i_1}))$ of the midpoint $w_{i_1}$ of the segment $[x_ix_{i+1}]$ is in the region between $x_{i_1-1}$ and $x_{i_1+1}$. By the construction, in this region the convex body $K_a$ is circular, while $K_c$ is polygonal. Thus, the equality $r(PU(K_b), K_a) = s \cdot d_N$ implies now that $\|P(\widetilde{U}(w_{i_1}))\| \geq \frac{1}{s \cdot d_N},$ or equivalently $\frac{\|P(\widetilde{U}(w_{i_1}))\|}{\|\widetilde{U}(w_{i_1})\|} \geq \frac{1}{s}.$ Similarly we have$\frac{\|P(\widetilde{U}(w_{i_3}))\|}{\|\widetilde{U}(w_{i_3})\|} \geq \frac{1}{s}.$

Let us recall that we assumed that $u$ lies between $x_j$ and $x_{j+1}$ and $v$ lies between $x_{j+N}$ and $x_{j+N+1}$ or $v$ lies between $x_{j-N}$ and $x_{j-N+1}$. In either case, we see that, by the definition of the set of indices $I$ and the fact that $\beta \in \left [-\frac{\pi}{4N}, \frac{\pi}{4N} \right ]$, all three rays $\widetilde{U}(w_{i_1}), \widetilde{U}(w_{i_2}), \widetilde{U}(w_{i_3})$ are contained between the perpendicular rays of $u$ and $v$. However, for a unit vector $x \in \mathbb{R}^2$, the maximum of $\|P(x)\|$ is clearly equal to $A$ and it is attained only for $x = \pm u$. Similarly, the minimum of $\|P(x)\|$ is equal to $1$ and is attained for $x = \pm v$. Furthermore, it is easy to see that on the shorter arc between $u$ and $v$ the function $x \to \|P(x)\|$ is decreasing. Hence, since the inequality $\frac{\|P(x)\|}{\|x\|} \geq \frac{1}{s}$
holds for $x=\widetilde{U}(w_{i_1})$ and $x=\widetilde{U}(w_{i_3})$ it has to hold for every $x$ lying on a ray between the rays of $w_{i_1}$ and $w_{i_3}$. In particular, coming back to the notation of the previous step, this inequality is satisfied for the vector $\widetilde{U}(y)$ lying on the ray $\ell_1$, as it is contained in the region between $x_{i_2-1}$ and $x_{i_2+1}$. If we now repeat the previous step, but this time estimating as in \eqref{ratio} with the improved inequality for $\|P(\widetilde{U}(y))\|$, we obtain
\[\|z\| = \|P(\widetilde{U}(y))\| \geq \frac{1}{s} \cdot \|\widetilde{U}(y)\| = \frac{1}{s} \cdot \|y\|=\frac{1}{s}.\]
It follows that on the ray $\ell$ the ratio between points of the boundaries of $P(U(K_b))$ and $K_a$ is equal to $\frac{\|z\|}{\|w_{i_2}\|} \geq \frac{d_N}{s}$. This proves that $r(K_a, PU(K_b)) \geq \frac{d_N}{s}$ and ultimately the inequality \eqref{ineq} does not hold. This concludes the proof.
\end{proof}

In the remark below, we discuss in more detail the optimal order of the construction.

\begin{remark}
\label{optimal}
As mentioned in the introduction, Bronstein in \cite{bronstein2} has proved that the maximum cardinality of a $(1+\varepsilon)$-separated set in the $n$-dimensional Banach--Mazur compactum is of the order $\exp (c\varepsilon^{\frac{1-n}{2}})$, which for $n=2$ matches our estimate for the cardinality of an equilateral set and shows that our construction is essentially optimal. However, the paper \cite{bronstein2} does not seem to be easily accessible. For the sake of completeness, we shall shortly explain optimality of the construction, based on
an earlier paper of Bronstein \cite{bronstein}, which is easily found on-line.

Specifically, we will refer to Theorem $3$ from that paper, which states that for every sufficiently small $\varepsilon>0$ there exists an $\varepsilon$-net of the cardinality $C^{\varepsilon^{\frac{1-n}{2}}}$ (where $C>1$ is a constant depending on $n$) for the space of convex, closed sets of the $n$-dimensional Euclidean unit ball $\mathbb{B}_n$, when equipped with the Hausdorff distance $d_H$. In particular, this cardinality is an upper bound for a cardinality of $2\varepsilon$-separated set in this space. To transfer this to the setting of the Banach--Mazur distance, let us suppose that symmetric convex bodies $C_1, \ldots, C_N \in \mathbb{R}^n$ form a $(1+\varepsilon)$-separated set in the $d_{BM}$ metric. By the John's Ellipsoid Theorem, we can suppose, after applying a suitable linear transformation, that $\frac{1}{\sqrt{n}} \mathbb{B}_n \subseteq C_i \subseteq \mathbb{B}_n$ for every $1 \leq i \leq N$. We note that in this case for any $1 \leq i < j \leq N$ we have
\[C_i \subseteq C_j + d_H(C_i, C_j) \mathbb{B}_n \subseteq C_j + \sqrt{n}d_H(C_i, C_j) C_j = (1 + \sqrt{n} d_H(C_i, C_j))C_j\]
and since a similar inequality holds for $i, j$ swapped we deduce that
\[1 + \varepsilon \leq d_{BM}(C_i, C_j) \leq \left ( 1 + \sqrt{n} d_H(C_i, C_j) \right )^2.\]
Hence
\[d_H(C_i, C_j) \geq \frac{\sqrt{1+\varepsilon}-1}{\sqrt{n}} \geq \alpha\varepsilon,\]
for some constant $\alpha>0$. Thus, convex bodies $C_i$ form an $(\alpha \varepsilon)$-separated set also in the Hausdorff distance and in conclusion it follows that $N \leq C_0^{\varepsilon^{\frac{1-n}{2}}}$ for some constant $C_0$ (depending on $n$). 

For $n=2$, we thus obtain that the maximum cardinalities of $d^2_N$-equilateral and $d^2_N$-separated sets are both between $\exp\left (c_1\sqrt{\varepsilon_{N}^{-1}} \right )$ and $\exp\left ( c_2\sqrt{\varepsilon_{N}^{-1}} \right )$ (where $d^2_N = 1 + \varepsilon_N$) for some constants $c_2>c_1>0$. However, it should be noted that our construction is most likely not optimal in the sense that it works only for a discrete set of distances $d^2_N$. In other words, nothing can be deduced from our result about $d$-equilateral sets for $d \neq d^2_N$.

\end{remark}

\section{Concluding remarks}

It is not clear if our construction can be generalized to higher dimensions. The regular polygon, lying in the core of our construction, has some exceptionally good properties with no clear analogue in any of the other dimensions. Its vertices form a good $\varepsilon$-net for the circle, but additionally it also has many automorphisms, which allows us to closely approximate any linear transformation with one that fixes the polygon. Nevertheless, most likely the $n$-dimensional Banach--Mazur compactum contains arbitrarily large equilateral sets for any dimension $n \geq 2$. It is also quite possible, that similarly as in Theorem \ref{mainthm}, the equilateral sets are roughly as large as the separated sets of the same distance.

There are many other questions that could be asked, even in the planar setting. For example, could our construction be improved to work not only for a discrete set of distances but for all distances in the interval $(1, c]$ for some constant $c>1$? Furthermore, taking into the account the fact that the maximum size of an equilateral set matches the maximum size of a separated set in the planar Banach--Mazur compactum, one may expect that the Banach--Mazur compactum could possibly have some kind of a universality property for finite metric spaces. That is, that the Banach--Mazur compactum may actually contain all finite metric spaces, assuming that the distances are allowed to be scaled close to $1$. 

\bibliographystyle{amsplain}

\end{document}